\documentclass[draft]{amsart}
\usepackage{amssymb,graphicx,enumerate,amsthm}
\usepackage{multirow}
\usepackage{xcolor}
\usepackage{cite}

\newcommand{\T}{\mathcal{T}}
\newcommand{\Z}{\mathbb{Z}}

\numberwithin{equation}{section}
\newtheorem{theorem}{Theorem}[section]
\newtheorem{lemma}[theorem]{Lemma}

\begin{document}

\makeatletter
\def\imod#1{\allowbreak\mkern10mu({\operator@font mod}\,\,#1)}
\makeatother

\author{Alexander Berkovich}
   	\address{Department of Mathematics, University of Florida, 358 Little Hall, Gainesville FL 32611, USA}
   	\email{alexb@ufl.edu}

\author{Ali Kemal Uncu}
   \address{Research Institute for Symbolic Computation, Johannes Kepler University, Linz. Altenbergerstrasse 69
A-4040 Linz, Austria}
   \email{akuncu@risc.jku.at}

\thanks{Research of the first author is partly supported by the Simons Foundation, Award ID: 308929. Research of the second author is supported by the Austrian Science Fund FWF, SFB50-07 and SFB50-09 Projects.}

\title[\scalebox{.9}{Elementary Polynomial Identities Involving \MakeLowercase{$q$}-Trinomial Coefficients}]{Elementary Polynomial Identities Involving \MakeLowercase{$q$}-Trinomial Coefficients}
     
\begin{abstract} 
We use $q$-binomial theorem to prove three new polynomial identities involving $q$-trinomial coefficients. We then use summation formulas for the $q$-trinomial coefficients to convert our identities into another set of three polynomial identities, which imply Capparelli's partition theorems when the degree of the polynomial tends to infinity. This way we also obtain an interesting new result for the sum of the Capparelli's products. We finish this paper by proposing an infinite hierarchy of polynomial identities.
\end{abstract}

  


\date{\today}
   
\dedicatory{To George E. Andrews, on the occasion of his 80th birthday}
   
\maketitle

\section{Introduction}

George Andrews is known for his many accomplishments and impeccable leadership in both his mathematical contributions and in service to the community of mathematics. His influence on research keeps opening new horizons, and at the same time, new doors to young researchers. There are many areas of study he introduced that are now saturated with world-class mathematicians, yet there are many more that the community is only catching up studying. Here, we look at one of these lesser-studied objects: $q$-trinomial coefficients. Introduced by Andrews in collaboration with Baxter, $q$-trinomial coefficients defined by 
\begin{align}\label{Round_Tri_Def}
\left(\begin{array}{c}L,\, b\\ a \end{array};q\right)_2 &:= \sum_{n\geq 0 } q^{n(n+b)} \frac{(q;q)_L}{(q;q)_n(q;q)_{n+a}(q;q)_{L-2n-a}},\\ \intertext{where, for any non-negative integer $n$, $(a,q)_n$ is the standard $q$-Pochhammer symbol \cite{Theory_of_Partitions},}
\nonumber(a;q)_n &:= (1-a)(1-aq)(1-aq^2)\dots(1-aq^{n-1}).
\end{align}

It is easy to verify that \[\sum_{a=-L}^L \left(\begin{array}{c}L,\, b\\ a \end{array};1\right)_2 t^a = \left(t+1+\frac{1}{t}\right)^L,\] 
which implies the generalized Pascal Triangle for \eqref{Round_Tri_Def} with $q=1$:
\[\begin{array}{ccccccccccc}
• &• & • & • & • & 1 & • & • & • & • \\ 
• &• & • & • & 1 & 1 & 1 & • & • & • \\ 
• &• & • & 1 & 2 & 3 & 2 & 1 & • & • \\ 
• &• & 1 & 3 & 6 & 7 & 6 & 3 & 1 & • \\
• & 1& 4 & 10& 16&19& 16&10&4& 1 & • \\ 
\reflectbox{$\ddots$} & \vdots& \vdots& \vdots & \vdots & \vdots & \vdots & \vdots & \vdots & \vdots & \ddots
\end{array} \]

The $q$-trinomial coefficients were studied in \cite{Andrews_P_Eulers,Andrews_P_Capparelli, Andrews_P_RR, Andrews_P_q_Tri, Andrews_Baxter, Andrews_Berkovich, Berkovich_McCoy_Pearce, Berkovich_McCoy_Orrick, BerkovichUncu7, Warnaar_Note, Warnaar, Warnaar_Refined, Warnaar_T}. Nevertheless, it appears that the following identities are new.

\begin{theorem}\label{Tri_Summation_Formulas_THM}
\begin{align}
\nonumber \sum_{n\geq 0} (-1)^n q^{(3n^2+n)/2}&\frac{(q^3;q^3)_L}{(q;q)_{L-2n}(q^3;q^3)_n}+ q^{2L+1}\sum_{n\geq 0} (-1)^n q^{(3n^2-n)/2} \frac{(q^3;q^3)_L}{(q;q)_{L-2n}(q^3;q^3)_n}\\ 
\label{First_Pair}&= \sum_{j=-L}^L \left\{q^{L+j+1}\left(\begin{array}{c}L,\, j+1\\ j \end{array};q^3\right)_2+ q^{L+4j}\left(\begin{array}{c}L,\, j\\ j-1 \end{array};q^3\right)_2\right\},\\
\label{Second_Pair} \sum_{n\geq 0} (-1)^n q^{(3n^2-n)/2}& \frac{(q^3;q^3)_L}{(q;q)_{L-2n}(q^3;q^3)_n}  =  \sum_{j=-L}^L  q^{2L-j}\left(\begin{array}{c}L,\, j-1\\ j \end{array};q^3\right)_2 ,\\
\label{Third_Pair} \sum_{n\geq 0} (-1)^n q^{(3n^2+n)/2}& \frac{(q^3;q^3)_L}{(q;q)_{L-2n}(q^3;q^3)_n}  =  \sum_{j=-L}^L   q^{L-j}\left(\begin{array}{c}L,\, j\\ j \end{array};q^3\right)_2 .
\end{align}
\end{theorem}

For $|q|<1$, it is not obvious how the right-hand sides of \eqref{First_Pair}-\eqref{Third_Pair} behave as $L$ tends to $\infty$. However, from the left-hand sides we can easily discover the limiting formulas \[\frac{1}{(q;q^3)_\infty},\, \frac{1}{(q^2;q^3)_\infty},\text{  and  }\frac{1}{(q;q^3)_\infty}, \] respectively, as $L\rightarrow\infty$ with the aid of the $q$-binomial theorem.

These identities are related to combinatorics and partition theory. As an example, we will show that \eqref{First_Pair} and \eqref{Third_Pair} in Theorem~\ref{Tri_Summation_Formulas_THM} imply Capparelli's partition theorems. Moreover, the same application will yield the following new interesting result.

\begin{theorem}\label{Thm_Outlook2}
\begin{equation}\label{EQN_Outlook2}
\sum_{m,n\geq 0} \frac{q^{2m^2 + 6mn +6n^2 -2m-3n}}{(q;q)_m(q^3;q^3)_n} = (-q^2,-q^4;q^6)_\infty(-q^3;q^3)_\infty  + (-q,-q^5;q^6)_\infty(-q^3;q^3)_\infty .
\end{equation}
\end{theorem}

In Section~\ref{Sec_Background}, we give a comprehensive list of definitions and identities that will be used in this paper. Section~\ref{Sec_Tri} has the proof of Theorem~\ref{Tri_Summation_Formulas_THM}. This section also includes the dual of these identities and a necessary version of the Bailey lemma for the $q$-trinomial coefficients. We find new polynomial identities that yield Capparelli's partition theorem in Section~\ref{Sec_Cap}. Theorem~\ref{Thm_Outlook2}, which includes the sum of two the Capparelli's theorem's products, is also proven in Section~\ref{Sec_Cap}. This section also contains a comparison of the mentioned polynomial identities and the previously found polynomial identities \cite{BerkovichUncu7} that also imply the Capparelli's partition theorems. Outlook section, Section~\ref{Section_Outlook}, includes two new results the authors are planning on presenting soon: a doubly bounded identity involving Warnaar's refinement of the $q$-trinomial coefficients, and also an infinite hierarchy of $q$-series identities.

\section{Necessary Definitions and Identities}\label{Sec_Background}

We use the standard notations as in \cite{Theory_of_Partitions}. For formal variables $a_i$ and $q$, and a non-negative integer $N$ 
\begin{align}(a;q)_\infty&:= \lim_{n\rightarrow\infty}(a;q)_n,\\
(a_1,a_2,\dots,a_k;q)_n &:= (a_1;q)_n(a_2;q)_n\dots (a_k;q)_n\text{ for any}n\in\Z_{\geq0}\cup\{\infty\}.\\
\intertext{We can extend the definition of $q$-Pochhamer symbols to negative $n$ with}
\label{Shifting_of_infinite_Products} (a;q)_n &= \frac{(a;q)_\infty}{(aq^n;q)_\infty}.\\
\intertext{Observe that \eqref{Shifting_of_infinite_Products} implies}
\label{One_over_neg_k} \frac{1}{(q;q)_n} &= 0\text{  if  }n <0.\\
\intertext{Also observe that, for non-negative $n$ we have}
\label{qPoch_q_to_one_over_q}\displaystyle (q^{-1};q^{-1})_n &= (-1)^n q^{n+1\choose 2} (q;q)_n .
\end{align}
We define the $q$-binomial coefficients in the classical manner as
	\begin{align}
 \label{Binom_def}
  \displaystyle {m+n \brack m}_q &:= \left\lbrace \begin{array}{ll}\frac{(q)_{m+n}}{(q)_m(q)_{n}},&\text{for }m, n \geq 0,\\
   0,&\text{otherwise.}\end{array}\right.
\intertext{It is well known that for $m\in\mathbb{Z}_{\geq0}$}
 \label{Binom_limit}
 \displaystyle \lim_{N\rightarrow\infty}{N\brack m}_q &= \frac{1}{(q;q)_m},
 \intertext{for any $j\in \mathbb{Z}_{\geq0}$ and $\nu =0$ or 1}
  \label{Binom_limit2} \displaystyle \lim_{M\rightarrow\infty}{2M+\nu\brack M-j}_q &= \frac{1}{(q;q)_\infty}.
\end{align}

We define another $q$-trinomial coefficient for any integer $n$:
\begin{align}
\label{T_n_Def} T_n\left( \begin{array}{c}L\\a\end{array};q\right) &:= q^{(L(L-n)-a(a-n))/2} \left(\begin{array}{c}L,\, a-n\\ a \end{array};\frac{1}{q}\right)_2.
\end{align}

\begin{theorem}[q-Binomial Theorem]\label{qBin_THM} For variables $a,\, q,$ and $z$,
\begin{equation}\label{qBin_EQN}\sum_{k\geq 0} \frac{(a;q)_n}{(q;q)_n}z^n = \frac{(az;q)_\infty}{(z;q)_\infty}.\end{equation}
\end{theorem}

Also note that the $q$-exponential sum \begin{equation}
\label{qExonential_sum} \sum_{k\geq 0} \frac{q^{n(n-1)/2}}{(q;q)_n}z^n = (-z;q)_\infty
\end{equation} is a limiting case ($a\rightarrow \infty$ after the variable change $z\mapsto -z/a$) of \eqref{qBin_EQN}.

Another ingredient we will use here is the Jacobi Triple Product Identity \cite{Theory_of_Partitions}
\begin{theorem}[Jacobi Triple Product Identity]
\begin{equation}\label{JTP}
\sum_{j=-\infty}^\infty z^j q^{j^2} = (q^2,-zq,-\frac{q}{z};q^2)_\infty.
\end{equation}
\end{theorem}

\section{Proof of Theorem~\ref{Tri_Summation_Formulas_THM} and some $q$-Trinomial Summation formulas}\label{Sec_Tri}

We start with the following lemma.

\begin{lemma}\label{Round_Tri_tri_var_GF_THM} For any integer $n$, we have 
\begin{equation}\label{Round_Tri_tri_var_GF_EQN}
\sum_{L\geq 0}\sum_{j=-\infty}^\infty  \frac{x^j t^L }{(q;q)_L} \left(\begin{array}{c}L,\, j-n\\ j \end{array};q\right)_2 = \frac{(t^2 q^{-n} ; q)_\infty}{(t, \frac{t}{x} q^{-n},tx;q)_\infty}.
\end{equation}
\end{lemma}

The $n=0$ case of \eqref{Round_Tri_tri_var_GF_EQN} first appeared in Andrews \cite[p.153, (6.6)]{Andrews_P_Capparelli}.

\begin{proof}
We start by writing the definition \eqref{Round_Tri_Def} on left-hand side of the formula \eqref{Round_Tri_tri_var_GF_EQN}. After a simple cancellation one sees that the triple sum can be untangled by the change of the summation variables $\nu = k+j$, and $\mu=L-2k-j$. This change of summation variables, keeping \eqref{One_over_neg_k} in mind, shows that the left-hand side sum of \eqref{Round_Tri_tri_var_GF_EQN} can be written as: \begin{equation}\label{triple sum}
\sum_{k\geq 0} \frac{x^{-k}t^k q^{-nk}}{(q;q)_k}\sum_{\nu\geq 0} \frac{x^\nu t^\nu q^{k\nu}}{(q;q)_\nu} \sum_{\mu\geq 0} \frac{t^{\mu}}{(q;q)_\mu}.
\end{equation} One can apply the $q$-Binomial Theorem starting from the innermost sum of \eqref{triple sum}. After applying the $q$-Binomial Theorem~\ref{qBin_THM} with $(a,z)= (0,t)$, and $(a,z)=(0,xtq^k)$ we get 
\begin{equation}\label{triple_sum_2}
\frac{1}{(t;q)_\infty}\sum_{k\geq 0} \frac{x^{-k}t^k q^{-nk}}{(q;q)_k (xtq^k;q)_\infty}.
\end{equation} We rewrite $(xtq^k;q)_\infty$ using \eqref{Shifting_of_infinite_Products}, take the $k$-free portion out of the summation, and use \eqref{qBin_EQN} once again with $(a,z) = (xt,x^{-1}tq^n)$ to finish the proof.
\end{proof}

We can prove Theorem~\ref{Tri_Summation_Formulas_THM} using Lemma~\ref{Round_Tri_tri_var_GF_THM}.

\begin{proof}[Proof of Theorem~\ref{Tri_Summation_Formulas_THM}] Instead of proving these identities directly, we will prove the equality of their generating functions. It is clear that one can prove the equality of the two sides of polynomial identities of the form \[A_L(q) = B_L(q)\] by a multi-variable generating function equivalence \begin{equation}\label{Bi_variate}\sum_{L\geq 0} \frac{t^L}{(q^3;q^3)_L} A_{L}(q) = \sum_{L\geq 0} \frac{t^L}{(q^3;q^3)_L} B_{L}(q). \end{equation}

On the right-hand side of \eqref{Bi_variate} with the choice of $B_{L}(q)$ being the right-hand sides of \eqref{First_Pair}-\eqref{Third_Pair}, we get \begin{equation}\label{summed_three_products}\frac{(t^2q^2;q^3)_\infty}{(t;q)_\infty}\frac{(1+q)}{(1+tq)},\text{  } \frac{(t^2q;q^3)_\infty}{(t;q)_\infty},\text{  and  }\frac{(t^2q^2;q^3)_\infty}{(t;q)_\infty},\end{equation} respectively, by Lemma~\ref{Round_Tri_tri_var_GF_THM}. Hence, all we need to do is to show that the left-hand side of \eqref{Bi_variate} with the choice of $A_L(q)$ being the left-hand sides of \eqref{First_Pair}-\eqref{Third_Pair} yields the same products.

The left-hand side of \eqref{First_Pair} has two sums. The first sum of the left-hand side of \eqref{First_Pair} after being multiplied by $t^L / (q^3;q^3)_L$, summing over $L$ as suggested in \eqref{Bi_variate}, and after simple cancellations turns into \begin{equation} \label{First_pair_sum1_raw}
\sum_{L,n\geq 0} (-1)^n \frac{q^{\frac{3n^2+n}{2}}t^L}{(q;q)_{L-2n}(q^3;q^3)n}.
\end{equation} We introduce the new summation variable $\nu = L-2n$. This factors the double sum fully. Keeping \eqref{One_over_neg_k} in mind, we rewrite \eqref{First_pair_sum1_raw} as \[\sum_{\nu\geq 0} \frac{t^\nu}{(q;q)_\nu}\sum_{n\geq 0}\frac{q^{3n(n-1)/2}}{(q^3;q^3)_n}(-t^2q)^n.\] Then, using \eqref{qBin_EQN} and \eqref{qExonential_sum} on the two sums, respectively, we see that \eqref{First_pair_sum1_raw} is equal to \[\frac{(t^2q;q^3)_\infty}{(t;q)_\infty}.\]
The same exact calculation can be done for the second sum on the left-hand side of \eqref{First_Pair}, and the left-hand side sums of \eqref{Second_Pair} and \eqref{Third_Pair}. After the simplifications we see that the products we get from the left-hand side sums after \eqref{Bi_variate} is applied to them, are the same as the products \eqref{summed_three_products}. 
\end{proof}

In the identities of Theorem~\ref{Tri_Summation_Formulas_THM}, we replace $q\mapsto 1/q$, multiply both sides of the equations by $q^{3L^2/2}$, use \eqref{qPoch_q_to_one_over_q} and \eqref{T_n_Def}, and do elementary simplifications to get the following theorem.

\begin{theorem}\label{Tri_Summation_DUAL_Formulas_THM}
\begin{align}
\nonumber \sum_{n\geq 0} q^{{L-2n\choose 2}} &\frac{(q^3;q^3)_L}{(q;q)_{L-2n}(q^3;q^3)_n} + q^{L+1}\sum_{n\geq 0} q^{{L-2n+1 \choose 2}+n} \frac{(q^3;q^3)_L}{(q;q)_{L-2n}(q^3;q^3)_n}\\
\label{First_pair_dual} &\hspace{1cm}= \sum_{j\geq 0} q^{\frac{3j^2+j}{2}}\left\{T_{-1}\left(\begin{array}{c}L\\ j \end{array};q^3 \right)  +T_{-1}\left(\begin{array}{c}L\\ j+1 \end{array};q^3 \right) \right\},\\
\label{Second_pair_dual} \sum_{n\geq 0} q^{L-2n\choose 2} &\frac{(q^3;q^3)_L}{(q;q)_{L-2n}(q^3;q^3)_n} = \sum_{j\geq 0} q^{\frac{3j^2-j}{2}} T_{1}\left(\begin{array}{c}L\\ j \end{array};q^3 \right),\\
\label{Third_pair_dual}\sum_{n\geq 0} q^{\frac{(L-2n)^2}{2}} &\frac{(q^3;q^3)_L}{(q;q)_{L-2n}(q^3;q^3)_n} = \sum_{j\geq 0} q^{\frac{3j^2+2j}{2}} T_{0}\left(\begin{array}{c}L\\ j \end{array};q^3 \right).
\end{align}
\end{theorem}

Building on the development in \cite{Andrews_Berkovich}, \cite{Berkovich_McCoy_Pearce}, and \cite{Warnaar_Note}, Warnaar \cite[eqns. (10),(14)]{Warnaar} proved the following summation formulas.
\begin{theorem}[Warnaar]
\begin{align}
\label{T0_BL} &\sum_{i\geq 0} q^{\frac{i^2}{2}} {L\brack i}_q T_0\left(\begin{array}{c}i\\a\end{array};q\right) = q^{\frac{a^2}{2}} {2L\brack L-a}_q,\\
\label{T1_BL}&\sum_{i\geq 0} q^{{i\choose 2}} (1+q^L){L\brack i}_q T_1\left(\begin{array}{c}i\\ a \end{array};q \right) = (1+q^a) q^{{a \choose 2}} {2L \brack L-a}_q.
\end{align}
\end{theorem}
We found a new similar summation formula:
\begin{theorem}
\begin{equation}
\label{T_min_1_BL} \sum_{i\geq 0} q^{{i+1 \choose 2}} {L\brack i}_q \left\{ T_{-1}\left(\begin{array}{c}i\\ a \end{array};q \right) + T_{-1}\left(\begin{array}{c}i\\ a+1 \end{array};q \right)  \right\} = q^{a+1\choose 2} {2L+1\brack L-a}_q.
\end{equation}
\end{theorem}

\begin{proof} To prove \eqref{T_min_1_BL}, we need the following identity of Berkovich--McCoy--Orrick \cite[p. 815, (4.8)]{Berkovich_McCoy_Orrick}:
\begin{equation}\label{T_min_1_Transform}
T_{-1}\left(\begin{array}{c}L\\ a \end{array};q \right) + T_{-1}\left(\begin{array}{c}L\\ a+1 \end{array};q \right) = \frac{1}{1-q^{L+1}}\left\{T_{1}\left(\begin{array}{c}L+1\\ a \end{array};q \right) -q^{(L+1-a)/2} T_{0}\left(\begin{array}{c}L+1\\ a \end{array};q \right) \right\}.
\end{equation}
After the use of \eqref{T_min_1_Transform} on the left-hand side of \eqref{T_min_1_BL}, we employ \begin{equation}\label{qBinom_shift}\frac{1}{1-q^{i+1}}{L\brack i}_q = \frac{1}{1-q^{L+1}}{L+1\brack i+1}_q,\end{equation} and summations \eqref{T0_BL} and \eqref{T1_BL}. This yields the right-hand side of \eqref{T_min_1_BL} after some elementary simplifications.
\end{proof}

\begin{theorem} Let $F_i(L)$ and $\alpha_i(a)$ be sequences, depending on $L$ and $a$, respectively, for $i=-1,\, 0$ or $1$. If
 \begin{align} 
 \label{F_0_def} F_0(L) &= \sum_{a=-\infty}^\infty \alpha_0(a) T_0\left(\begin{array}{c}L\\ a \end{array};q \right), \\
 \label{F_1_def} F_1(L) &= \sum_{a=-\infty}^\infty \alpha_1(a) T_1\left(\begin{array}{c}L\\ a \end{array};q \right), \\
 \label{F_-1_def} F_{-1}(L) &= \sum_{a=-\infty}^\infty \alpha_{-1}(a) \left\{T_{-1}\left(\begin{array}{c}L\\ a \end{array};q \right)+T_{-1}\left(\begin{array}{c}L\\ a+1 \end{array};q \right)\right\},
\intertext{then}
\label{F_0_sum}\sum_{i\geq 0} q^{\frac{i^2}{2}} &{L \brack i}_q F_0(i) = \sum_{a=-\infty}^\infty \alpha_0(a) q^{\frac{a^2}{2}} {2L\brack L-a}_q,\\
\label{F_1_sum}(1+q^L)\sum_{i\geq 0} q^{i\choose 2} &{L \brack i}_q F_1(i) = \sum_{a=-\infty}^\infty \alpha_1(a)(1+q^a) q^{a \choose 2} {2L\brack L-a}_q,\\
\label{F_-1_sum}\sum_{i\geq 0} q^{i+1 \choose 2} &{L \brack i}_q F_{-1}(i) = \sum_{a=-\infty}^\infty \alpha_{-1}(a) q^{i+1\choose 2} {2L+1 \brack L-a}_q
\end{align} hold.
\end{theorem}

\begin{proof} We apply \eqref{T0_BL}-\eqref{T_min_1_BL} to \eqref{F_0_def}-\eqref{F_-1_def} and get \eqref{F_0_sum}-\eqref{F_-1_sum}, respectively.
\end{proof}

\section{New polynomial identities implying  Capparelli's partition theorems}\label{Sec_Cap}

We apply \eqref{F_0_sum} to \eqref{Third_pair_dual} to get

\begin{equation}\label{Capparelli_eqn1_raw}
\sum_{L,n\geq 0} q^{\frac{(L-2n)^2+3L^2}{2}} \frac{(q^3;q^3)_M}{(q;q)_{L-2n} (q^3;q^3)_n(q^3;q^3)_{M-L}} = \sum_{j=-M}^M q^{3j^2+j} {2M\brack M+j}_{q^3}.
\end{equation}

We introduce the new variable $m=L-2n$, and let \[Q(m,n):=2m^2 + 6mn +6n^2\] and, observe that\[Q(m,n)=\frac{(L-2n)^2+3L^2}{2},\] after the change of variable.
Hence, \eqref{Capparelli_eqn1_raw} can be written as
\begin{theorem}\label{FinCap1M}
\begin{equation}\label{THM71}
\sum_{m,n\geq 0} \frac{q^{Q(m,n)}(q^3;q^3)_M}{(q;q)_m (q^3;q^3)_n(q^3;q^3)_{M-2n-m}} = \sum_{j=-M}^M q^{3j^2+j} {2M\brack M+j}_{q^3}.
\end{equation}
\end{theorem}

Recall that \eqref{THM71} is Theorem~7.1 in \cite{BerkovichUncu7}. Letting $M\rightarrow \infty$ in \eqref{THM71}, using \eqref{Binom_limit2}, and the Jacobi Triple Identity \eqref{JTP} on the right-hand side we get 
\begin{theorem}
\begin{equation}\label{KR1}
\sum_{m,n\geq 0} \frac{q^{Q(m,n)}}{(q;q)_m (q^3;q^3)_n} =  (-q^2,-q^4;q^6)_\infty(-q^3;q^3)_\infty.
\end{equation} 
\end{theorem}

Authors recently discovered another polynomial identity \cite[Thm 1.3, (1.12)]{BerkovichUncu7} that imply the same $q$-series identity  \eqref{KR1} as $N\rightarrow\infty$:
\begin{theorem}\label{FinCap1N} For any non-negative integer $N$, we have
\begin{align*} 
&\sum_{m,n\geq 0} q^{Q(m,n)}{3(N-2n-m)\brack m}_q {2(N-2n-m)+n\brack n}_{q^3}= \sum_{l=0}^{N} q^{3{N-2l \choose 2}} {N\brack 2l}_{q^3} (-q^2,-q^4;q^6)_l.
\end{align*}
\end{theorem}

Recently, the identity \eqref{KR1} was independently proposed by Kanade--Russell \cite{Kanade_Russell} and Kur\c{s}ung\"oz \cite{Kagan2}.  They showed that \eqref{KR1} is equivalent to the following partition theorem.

\begin{theorem}[Capparelli's First Partition Theorem \cite{Capparelli_proof}]\label{m1_Capparelli} 
For any integer $n$, the number of partitions of $n$ into distinct parts where no part is congruent to $\pm 1$ modulo $6$ is equal to the number of partitions of $n$ into parts, not equal to $1$, where the minimal difference between consecutive parts is 2. In fact, the difference between consecutive parts is greater than or equal to $4$ unless consecutive parts are $3k$ and $3k+3$ (yielding a difference of 3), or $3k-1$ and $3k+1$ (yielding a difference of 2) for some $k\in \mathbb{N}$.
\end{theorem}

Theorem~\ref{m1_Capparelli} was first proven by Andrews in \cite{Andrews_P_Capparelli}. 

Analogously, we apply \eqref{F_1_sum} to \eqref{Second_pair_dual} and get 
\begin{theorem}
\begin{equation}\label{THM72}
\sum_{m,n\geq 0} \frac{q^{Q(m,n)-2m-3n}(q^3;q^3)_M}{(q;q)_m (q^3;q^3)_n(q^3;q^3)_{M-2n-m}}(1+q^{3M}) = \sum_{j=-M}^M q^{3j^2-2j}(1+q^{3j}) {2M\brack M+j}_{q^3}.
\end{equation}
\end{theorem}
Letting $M$ tend to infinity, and using \eqref{Binom_limit2} and \eqref{JTP} on the right-hand side proves Theorem~\ref{Thm_Outlook2}.

Similar to the above calculations, we apply \eqref{F_-1_sum} to \eqref{First_pair_dual} and get 
\begin{theorem}\label{FinCap2M}
\begin{align}
\nonumber\sum_{m,n\geq 0} \frac{q^{Q(m,n)+m+3n}(q^3;q^3)_M}{(q;q)_m (q^3;q^3)_n(q^3;q^3)_{M-2n-m}} &+\sum_{m,n\geq 0} \frac{q^{Q(m,n)+3m+6n+1}(q^3;q^3)_M}{(q;q)_m (q^3;q^3)_n(q^3;q^3)_{M-2n-m}}\\\label{THM71_2_also_Cap2}&\hspace{4cm}= \sum_{j=-M-1}^M q^{3j^2+2j} {2M+1\brack M-j}_{q^3}.
\end{align}
\end{theorem}

Letting $M\rightarrow\infty$ and using \eqref{Binom_limit2} and \eqref{JTP} on the right-hand side, we get
\begin{theorem}
\begin{equation}
\label{Cap2}\sum_{m,n\geq 0} \frac{q^{Q(m,n)+m+3n}}{(q;q)_m(q^3;q^3)_n} + \sum_{m,n\geq 0}\frac{q^{Q(m,n)+3m+6n+1}}{(q;q)_m(q^3;q^3)_n} = (-q,-q^5;q^6)_\infty(-q^3;q^3)_\infty.
\end{equation}
\end{theorem}

It is instructive to compare \eqref{THM71_2_also_Cap2} with the following polynomial identity \cite[Thm 1.3, (1.12)]{BerkovichUncu7}, which also implies \eqref{Cap2} as $N\rightarrow\infty$:
\begin{theorem}\label{FinCap2N} For any non-negative integer $N$, we have
\begin{align*}
&\sum_{m,n\geq 0} q^{Q(m,n)+m+3n}{3(N-2n-m)+2\brack m}_q {2(N-2n-m)+n+1\brack n}_{q^3}\\ 
&\hspace{1cm}+ \sum_{m,n\geq 0} q^{Q(m,n)+3m+6n+1}{3(N-2n-m)\brack m}_q {2(N-2n-m)+n\brack n}_{q^3} \\
\nonumber&\hspace{1cm}= \sum_{l=0}^{N} q^{3{N-2l \choose 2}} {N+1\brack 2l+1}_{q^3} (-q;q^6)_{l+1}(-q^5;q^6)_l.
\end{align*}
\end{theorem}
 
We note that \eqref{Cap2} first appeared in Kur\c{s}ung\"oz \cite{Kagan2}. In fact, this is equivalent to the Cappareli's Second Partition theorem.

\begin{theorem}[Capparelli's Second Partition Theorem \cite{Capparelli_proof}]\label{m2_Capparelli} 
For any integer $n$, the number of partitions of $n$ into distinct parts where no part is congruent to $\pm 2$ modulo $6$ is equal to the number of partitions of $n$ into parts, not equal to $2$, where the minimal difference between consecutive parts is 2. In fact, the difference between consecutive parts is greater than or equal to $4$ unless consecutive parts are $3k$ and $3k+3$ (yielding a difference of 3), or $3k-1$ and $3k+1$ (yielding a difference of 2) for some $k\in \mathbb{N}$.
\end{theorem}

We note that Kur\c{s}ung\"oz \cite{Kagan2} showed the equivalence of \eqref{Cap2} to Theorem~\ref{m2_Capparelli}. On the other hand, Kanade--Russell \cite{Kanade_Russell} showed the equivalence of a slightly different (yet equivalent) double sum identity to the Capparelli's second partition theorem.

Comparing Theorems~\ref{FinCap1M} and \ref{FinCap1N}, and Theorems~\ref{FinCap2M} and \ref{FinCap2N}, we see that the identities proven here look somewhat simpler. On the other hand, the objects that appear on both sides of the identities from \cite{BerkovichUncu7} clearly come with combinatorial interpretations and are made up of objects with manifestly positive coefficients. It is not necessarily clear that the left-hand sides of \eqref{THM71}, \eqref{THM72}, and \eqref{THM71_2_also_Cap2} have positive coefficients at first sight. A combinatorial study of these objects as generating functions, which would also show the non-negativity of the coefficients of these polynomials, is a task for the future.

\section{Outlook}\label{Section_Outlook}

Identity \eqref{Third_pair_dual} is a special case $M\rightarrow\infty$ of the following doubly bounded identity.

\begin{theorem}\label{Outlook_THM1}
\begin{equation}
\sum_{\substack{m\geq 0,\\ L\equiv m \text{ (mod 2)}}} q^{\frac{m^2}{2}}{3M \brack m}_q {2M + \frac{L-m}{2}\brack 2M}_{q^3} = \sum_{j=-\infty}^{\infty} q^{\frac{3j^2+2j}{2}} \T\left( \begin{array}{c}L,\, M \\j,\, j\end{array};q^3  \right),
\end{equation}
where 
\begin{equation}\label{T_Warnaar}\T\left( \begin{array}{c}L,\, M \\a,\, b\end{array};q  \right) := \sum_{\substack{n\geq0, \\ L-a\equiv n \text{ (mod 2)}}} q^{\frac{n^2}{2}} {M\brack n}_q {M+b+\frac{L-a-n}{2} \brack M+b}_q {M-b +\frac{L+a-n}{2}\brack M-b}_q. \end{equation}
\end{theorem}

The refinement \eqref{T_Warnaar} of the $q$-Trinomial coefficients were first introduced by Warnaar \cite{Warnaar_Refined,Warnaar_T}. 

In the forthcoming paper, we will show that Theorem~\ref{Outlook_THM1} implies the following infinite hierarchy of identities.

\begin{theorem} Let $\nu$ be a positive integer, and let $N_k = n_k+n_{k+1}+\dots +n_{\nu}$, for $k=1,2,\dots, \nu$. Then,
\begin{align}
\nonumber\sum_{\substack{i,m,n_1,n_2,\dots,n_\nu\geq 0,\\ i+m \equiv N_1+N_2+\dots + N_\nu \text{ (mod 2)}}} &q^{\frac{m^2+3(i^2+N_1^2+N_2^2+\dots + N_\nu^2)}{2}} {L-N_1\brack i}_{q^3}{3n_\nu\brack m}_{q}\\
\label{Outlook_Big_EQN}&\hspace{-1cm}\times {2n_\nu + (i-N_1-N_2-\dots-N_\nu-m)/2\brack 2n_\nu}_{q^3} \prod_{j=1}^{\nu-1} {i- \sum_{k=1}^j N_k+n_j \brack n_j}_{q^3}\\
\nonumber&\hspace{-1cm}= \sum_{j=\infty} ^\infty q^{3{\nu+2\choose 2}j^2+j} \left(\begin{array}{c} L,\, (\nu+2)j\\ (\nu+2)j \end{array}; q^3 \right)_2.
\end{align}
\end{theorem}

\section{Acknowledgement}

Authors would like to thank the Algorithmic Combinatorics group of the Research Institute for Symbolic Computation, lead by Peter Paule, for their hospitality and for providing them the research environment, where this work has flourished.

Research of the first author is partly supported by the Simons Foundation, Award ID: 308929. Research of the second author is supported by the Austrian Science Fund FWF, SFB50-07 and SFB50-09 Projects.

\end{document}